\documentclass[11pt]{article}

\usepackage[margin=1in]{geometry}
\usepackage{amssymb, xcolor}
\usepackage{hyperref}

\usepackage[notref,notcite,color]{showkeys}
\definecolor{labelkey}{rgb}{0,0,1}

\usepackage{tikz}
 \usepackage{pgfplots}
\usepackage{graphicx,xcolor}
\usepackage{amsthm,amsfonts,euscript,amsmath,comment,slashed,here,todonotes}
\usepackage[affil-it]{authblk}
\usepackage{mathrsfs}

\usepackage{mathtools}

\mathtoolsset{showonlyrefs}

\usepackage{tikz}
\usetikzlibrary{shapes.geometric, arrows}

\tikzstyle{startstop} = [rectangle, rounded corners, minimum width=0cm, minimum height=1cm,text centered, draw=black, fill=red!30]
\tikzstyle{process} = [rectangle, minimum width=0cm, minimum height=0cm, text centered, draw=black]
\tikzstyle{split} = [rectangle]
\tikzstyle{arrow} = [thick,->] 
\tikzstyle{line} = [thick]

\newtheorem{theorem}{Theorem}[section]

\theoremstyle{definition}

\theoremstyle{remark}

\newcommand{\abs}[1]{\lvert#1\rvert}

\newcommand{\p}{\ensuremath{\partial}}

\newcommand\be{\begin{equation}}
\newcommand\ee{\end{equation}}
\newcommand\bea{\begin{eqnarray}}
\newcommand\eea{\end{eqnarray}}
\newcommand\bi{\begin{itemize}}
\newcommand\ei{\end{itemize}}
\newcommand\ben{\begin{enumerate}}
\newcommand\bena{\begin{enumerate}[(a)]}
\newcommand\een{\end{enumerate}}

\newcommand\bp{\begin{proof}}
\newcommand\ep{\end{proof}}

\allowdisplaybreaks

\title{Linear decay of the $\beta$-plane equation near Couette flow on the plane}
\author{Jacob Bedrossian\footnote{Department of Mathematics, University of California, Los Angeles, CA 90095, USA \texttt{jacob@math.ucla.edu}} \qquad Patrick Flynn\footnote{Department of Mathematics, University of California, Los Angeles,  CA 90095, USA, \texttt{pflynn@math.ucla.edu}} \qquad Sameer Iyer\footnote{Department of Mathematics, University of California, Davis, CA 95616, USA
\texttt{sameer@math.ucdavis.edu}}}

\begin{document}
\maketitle
\begin{abstract}
We prove new time decay estimates for the linearized $\beta$-plane equation near the Couette flow on the plane that combine inviscid damping and the dispersion of Rossby waves. Specifically, we show that the profiles of the velocity field components (i.e. $u(t,x+ty,y)$) decay pointwise on any compact set with polynomial rates. While mixing dominates for streamwise frequencies that are $O(1)$, dispersive effects need to be extracted for low streamwise frequencies that appear along a critical ray in frequency space. 
\end{abstract}

\setcounter{tocdepth}{2}
{\small\tableofcontents}

\section{Introduction}

In this paper, we study the 2D Euler equations with the $\beta$-plane effect, linearized around the Couette flow $(y,0)^T$ on $\mathbb R^2$.
The Euler equations with the $\beta$-plane effect is a canonical model in atmosphere and ocean sciences \cite{vallis2017atmospheric,pedlosky2013geophysical}, where it is sometimes called the \emph{barotropic vorticity equations}, governing barotropic dynamics in certain regimes such as Rossby waves for small data and some aspects of geostrophic turbulence for larger data \cite{rhines1979geostrophic,rhines1975waves}.    
Here, we are interested in understanding the interplay between the Rossby waves driven by the $\beta$-plane effect and the filamentation (inviscid damping) due to the Couette flow. 
It is known that Rossby waves will be absorbed by the shear flow; in geophysical fluid dynamics this is sometimes called Booker-Bretherton absorption \cite{booker1967critical}. The goal of this paper is to make a more quantitative study of this effect, endeavouring to obtain decay estimates that combine the inviscid damping and dispersive decay at low frequencies.

Denoting the (perturbation) vorticity $\omega(t,x,y)$, the linearized equations become
\begin{align}
&\p_t \omega + y \p_x \omega - {\p_x}\Delta_{x,y}^{-1} \omega = 0, 
\label{eq:2deuler}\\
&\omega(0,x,y) = \omega_0(x,y).
\end{align}
In the above,  we have re-scaled $t$, $x$, and $y$ to eliminate the parameters of the shear strength and $\beta$ itself. 
Formally, when the $x$-wavelength is small, the dominant term is the $\beta$-plane effect, which limits to a dispersive equation 
\begin{align*}
\partial_t q = {\p_x}\Delta_{x,y}^{-1} q.  
\end{align*}
The traveling waves of the above PDE are called \emph{Rossby waves}. Using $\xi$ and $\eta$ for the wavenumbers in $x$ and $y$ respectively, these traveling waves have the dispersion relation
\begin{align*}
  \Omega(\xi,\eta) = \frac{\xi}{\xi^2 + \eta^2}.
\end{align*}
These waves play an important role in a variety of geophysical fluid dynamics settings \cite{rhines1975waves,chelton1996global,pedlosky2013geophysical,vallis2017atmospheric}. 

On the other hand, if we eliminate the $\beta$-plane term, we are left with pure Couette flow 
\begin{align}
  & \p_t \omega + y \p_x \omega  = 0 \label{eq:CouetteNoBeta} \\ 
  & \omega(0,x,y) = \omega_0(x,y). 
\end{align}
By now it is well-known that this equation is dominated by inviscid damping; see e.g. \cite{BM15,bedrossian2019stability} for detailed discussions.
Inviscid damping of shear flows (and vortices) has recently been intensively studied in both linearized settings (see e.g. \cite{wei2020linearBeta,bedrossian2019vortex,wei2020linear,wei2018linear,jia2020linear,wei2019linear,yang2018linear,jia2025linearized}) and, to a lesser extent, nonlinear settings \cite{BM15,bedrossian2024uniform,bedrossian2023nonlinear,masmoudi2024nonlinear,ionescu2023nonlinear,zhao2025inviscid,deng2023long,lin2011inviscid,fan2025inertial}.   

We introduce the adapted coordinates 
\begin{align}
z = x- ty, \qquad f(t, z, y) = \omega(t, x, y), 
\end{align}  
which reduces \eqref{eq:2deuler} to
\begin{align}
\partial_t f - \frac{\p_z}{\p_z^2 + (\p_y - t \p_z)^2} f = 0. 
\end{align}
The variable $f$ is called the \emph{vorticity profile}. 
Taking the Fourier transform of $f$, defined as follows,
\[
\hat f(t,\xi,\eta) = \iint_{\mathbb R^2} e^{-i(\xi z + \eta y)} f(t,z,y) dz dy,
\]
we get the following equation for $\hat f$:
\begin{align}
\partial_t \widehat{f} + \frac{i \xi}{\xi^2 + (\eta - t \xi)^2} \widehat{f} = 0.  
\end{align}
Therefore, we have the closed form solution for $\hat f$:
\[
\hat f(t,\xi,\eta) = e^{i \Phi_t(\xi,\eta)} \hat \omega_0(\xi,\eta), 
\]
where 
\begin{align}
\Phi_t(\xi,\eta) = \frac{1}{\xi}\left( \arctan(t - \frac{\eta}{\xi}) + \arctan(\frac{\eta}{\xi})\right),
\end{align} 
yielding the following formula for the vorticity profile
\begin{align}\label{eq:vorticity_profile}
f(t,z,y)= \frac{1}{(2\pi)^2} \iint_{\mathbb R^2} \widehat{\omega}_0(\xi,\eta) e^{i (\xi z + \eta y +\Phi_t(\xi,\eta) )  }d\xi d\eta,
\end{align}
This integral then allows us to recover the \textit{stream function}  $\psi(t,x,y) = -\Delta_{x,y}^{-1} \omega$.
On the other hand, as in \cite{BM15}, to show time decay of $u$, it is more natural to study the profile of the stream function,
\begin{align}
\phi(t,z,y)= \frac{1}{(2\pi)^3} \iint_{\mathbb R^2} \widehat{\omega}_0(\xi,\eta) \frac{ e^{i (\xi z + \eta y +\Phi_t(\xi,\eta) )  }}{\xi^2 + (\eta - t\xi)^2} d\xi d\eta,
\end{align}
or, rather, its gradient:
\begin{align}
\phi_z (t,z,y) &= \frac{1}{(2\pi)^3} \iint_{\mathbb R^2}  {e^{i (\xi z + \eta y +\Phi_t(\xi,\eta) )}}\frac{i\xi}{\xi^2 + (\eta -t\xi)^2}\widehat{\omega}_0(\xi,\eta) d\xi d\eta, \label{eq:phi_z}\\
\phi_y (t,z,y) &= \frac{1}{(2\pi)^3} \iint_{\mathbb R^2}  {e^{i (\xi z + \eta y +\Phi_t(\xi,\eta) )  }}\frac{i\eta}{\xi^2 + (\eta -t\xi)^2}\widehat{\omega}_0(\xi,\eta) d\xi d\eta. \label{eq:phi_y}
\end{align}
These profiles allow us to recover the velocity field, since by the Biot-Savart law,
\[
u^x(t,x,y)= \phi_y(t,x+ty,y) -t\phi_z(t,x+ty,y), \quad \quad  u^y = \phi_z(t,x+ty,y).
\]
Thus, to show decay of $u$, it suffices to show decay of $\nabla_{z,y} \phi$. 

We can see from the integrals \eqref{eq:phi_z} and \eqref{eq:phi_y} that any plane wave initial data $\omega_0(x,y) = e^{i(x\xi + y\eta)}$ with $\xi\neq 0$ eventually mixes under the Couette shear flow. This is the inviscid damping of the Rossby waves. 
In fact, we can see that the inclusion of the $\beta$-plane term does not affect the arguments for deducing inviscid damping in \eqref{eq:CouetteNoBeta} and hence we have the streamfunction profile frequency-by-frequency damping, 
\begin{align*}
\abs{\hat{\phi}(\xi,\eta)} \lesssim \frac{1}{t^2} \left (\frac{1}{\xi^2} + \frac{\eta^2}{\xi^4}\right)|\hat \omega_0(\xi,\eta)|,
\end{align*}
which also gives similar decay estimates on $\nabla_{z,y} \phi$.
% (note each component of the velocity profile decays like $\abs{t\xi}^{-2}$, one an additional $\xi t$ on the $x$-component of the velocity due to the coordinate change). 
This inviscid damping provides strong decay for wavenumbers satisfying ${|\xi|}\geq 1 + |\eta|$. However, this effect is very weak at small $\xi$.
On the other hand, these low wavenumbers are exactly where we expect the dispersion to be most dominant. 
Our main theorem provides estimates which greatly improve on the raw inviscid damping estimates (at low frequencies and on compact sets) by harnessing the transient dispersive effects provided by the $\beta$-plane effect.

 Before we state our result, it is convenient to define some notation. We use the notational conventions
$\langle p \rangle := \sqrt{p^2 + 1} $, and write $A\lesssim B$ to mean $A \leq CB$ for some constant $C > 0$. We will also write $A \sim B$ to mean $B\lesssim A \lesssim B$. We use the norm
\[
\|\omega_0\|_Z =\sum_{j=0}^2 \|\langle \xi,\eta\rangle^6 \partial_\eta ^j\widehat \omega\|_{L^\infty_{\xi,\eta}}.
\]
Our main result is as follows. 
\begin{theorem}\label{thm:decay}
For all $t > 0$, we have
\begin{align}
|\phi_z(z,y)|& \lesssim \|\omega_0\|_Z \frac{\langle y\rangle^\frac{1}{2}} {\langle t\rangle} \\
 \quad |\phi_y(z,y)| &\lesssim  \|\omega_0\|_Z \frac{\langle y\rangle^\frac{1}{2} + \log(2 + t)} {\langle t\rangle^2} .
\end{align}
\end{theorem}
In the original coordinates, this provides pointwise decay like $\abs{u^x(t,x,y)} \lesssim_R \log (t + 2) \langle t\rangle^{-1/2}$ and $\abs{u^y(t,x,y)} \lesssim_R \log (t+2) \langle t\rangle^{-3/2}$ on sets $(x,y)$ such that $\frac{1}{R} \leq \abs{x - ty} + \abs{y} \leq R$. By comparison, we have the basic inviscid damping estimates, which are straightforward from \eqref{eq:phi_z} and \eqref{eq:phi_y}:
\begin{align}\label{eq:cheap_decay}
 \|(u^x,u^y)\|_{L^\infty} \lesssim \langle t\rangle\|\phi^z\|_{L^\infty} + \|\phi^y\|_{L^\infty} \lesssim \|\hat \omega_0\|_{L^1\cap L^\infty}.
\end{align}
% Alternatively, assuming $\hat \omega_0(0) =0$, we have the energy conservation identity
% \[
% \|(u^x,u^y)\| = \| \omega_0\|_{\dot H^{-1}} :=\left\|\frac{\hat \omega_0(\xi,\eta)}{|(\xi,\eta)|}\right\|_{L^2_{\xi,\eta}}. 
% \]
% However, more precise localization bounds are beyond the scope of this paper.

The additional decay from Theorem \ref{thm:decay} due to the low frequency dispersive effects can be compared to electrostatic plasma waves in the Vlasov-Poisson equations on $\mathbb R^3 \times \mathbb R^3$ \cite{bedrossian2022linearized,han2021linearized}. In that setting, one can use the Laplace transform to clearly isolate a (very) weakly damped electrostatic plasma wave with a dispersion relation matching that of warm plasmas (i.e. linearized Euler-Poisson, which there reduces to Klein-Gordon) and a residual part which is rapidly Landau damping. The reduction to Klein-Gordon at low frequencies then yields Schr\"odinger-like dispersive decay estimates. However, this method seems inapplicable here, and indeed, it is far from clear that a similarly distinct separation between the phase mixing and wave-like components of the vorticity exists at all. 
Nevertheless, one still sees an analogy, as in that setting one also has the low frequency decay of the electric field being far faster than that predicted by purely ballistic transport (i.e. the Landau damping present in kinetic free transport $\partial_t h + v \cdot \nabla_x h = 0$). 

On the other hand, much progress has been made in understanding the dispersive effect of rotation in the Euler equations for perturbations of $u=0$ in unbounded domains
\cite{elgindi2017,guo2023global,ko2024global,ko2025increased,pusateri2018global}.  For such problems, the situation is somewhat more tractable, since the linearization is a dispersive equation with constant coefficients. Nevertheless, as in this work, the operators considered in these works are strongly anisotropic, and require subtle decompositions in Fourier space using homogeneous coordinates. This approach is emulated here, particularly in how we handle the various boundary layers near the critical zone. We would also like to highlight the recent work \cite{fan2025inertial}, which shows instability for the nonlinear $\beta$-plane equation on $\mathbb R^3$.

\section{Proof of decay}

In this section, we prove the main theorem.\\

\noindent \textit{Proof of Theorem \ref{thm:decay}:} We break the proof up into a number of steps. Throughout, we assume $t \geq 1$ and $\|\omega_0\|_{Z} = 1$ without any loss of generality.\\

\noindent\textbf{Step 1} (\textit{setup}): It is convenient to work in the homogeneous coordinate
\[
\zeta = \frac{\eta}{\xi},
\]
instead of $\eta$.
We write $v(\xi,\zeta) = \mathrm{sgn}(\xi)\widehat \omega_0(\xi,\eta)$. Note that for $k = 0,1,2$ and $\zeta \in \mathbb R$,  we have
\begin{align}
\label{eq:aux_rescaling_estimate0}
|\partial_\zeta^{k} v(\xi,\zeta)| \leq \xi^k \langle \xi \langle \zeta\rangle\rangle^{-6}.
\end{align}
In particular,
\begin{align}\label{eq:aux_rescaling_estimate}
\int  | v(\xi,\zeta)| d\xi  \lesssim  \frac{\|\omega_0\|_{Z}}{\langle \zeta\rangle} = \frac{1}{\langle \zeta\rangle}.
\end{align}
We also define
\begin{align}
h_t(\zeta) &= \arctan(t-\zeta) + \arctan(\zeta) = \xi \Phi(\xi,\eta),\\
m^z_t(\zeta) &= \frac{1}{1+(\zeta- t)^2}, \\
m^y_t(\zeta) &= \frac{\zeta}{1+ (\zeta - t)^2}.
\end{align}
Under this reparametrization, 
 $\partial_w \phi$ for $w \in \{y,z\}$ reads
\begin{align}
\phi_w(z,y)  &= \frac{i}{(2\pi)^3} \iint_{\mathbb R^2}  {e^{i (\xi (z + \zeta y) +\frac{1}{\xi}h_t(\zeta) )}}m^w_t(\zeta) v(\xi,\zeta) d\xi d\zeta ,\label{eq:phi_w_reparam}.\end{align}
Generally, we suppress subscripts in $t$ for simplicity, i.e. $h = h_t$ and $m^w = m^w_t$.
To estimate the above, we use integration by parts in $\zeta$. 
Observe that
\begin{align}
 h'(\zeta)= \frac{1}{1 +\zeta^2} - \frac{1}{1 + (\zeta - t)^2} = \frac{t^2 - 2t \zeta }{(1 + \zeta^2)(1 + (\zeta - t)^2)}.
\end{align}
Hence, avoiding the line $2\zeta =  t$, we have
\begin{align}\label{eq:IBPphase}
e^{i \frac{1}{\xi} h(\zeta)} = \frac{\xi}{i h'(\zeta)} \partial_\zeta[ e^{i \frac{1}{\xi} h(\zeta)}].
\end{align}

We decompose the integral according to the size of $\frac{\zeta}{t}$. First, take $\chi \in C^\infty_c([-\frac{1}{4},\frac{1}{4}])$ with $\chi \equiv 1$ on $[-\frac{1}{8},\frac{1}{8}]$. Write
\[
\chi_{inner} (t,y,\zeta)=  \chi\left(\langle y\rangle^{\frac{1}{2}}\frac{\zeta -t}{t}\right), \qquad \chi_{outer} (t,y,\zeta) = 1-  \chi\left(\langle y\rangle^{\frac{1}{2}} \frac{\zeta -t}{t}\right)
\]
We then decompose our integral
\begin{align}
\phi_w (z,y)=\phi_{w,inner} + \phi_{w,outer} \coloneq   \iint_{\mathbb R^2} \chi_{inner} \ldots d\xi d\zeta + \iint_{\mathbb R^2} \chi_{outer} \ldots d\xi d\zeta.
\end{align}
We break the estimates on the inner and outer integrals into separate steps.\\

\noindent \textbf{Step 2} (\textit{outer integrals}): We apply \eqref{eq:aux_rescaling_estimate} to get
\[
|\phi_{z,outer}| \lesssim \int  \frac{\chi_{outer}}{\langle \zeta\rangle \langle \zeta -t \rangle^2} d\zeta, \qquad | \phi_{y,outer}| \lesssim \int  \frac{\zeta \chi_{outer}}{\langle \zeta\rangle \langle \zeta - t\rangle^2} d\zeta .
\]
On the support of $\chi_{outer}$, we have $|\zeta - t| > 0.125 t \langle y\rangle^{-\frac{1}{2}}$. Hence,
\begin{align}
&|\phi_{z,outer}|\\
  &\lesssim \int_{-\infty}^{-t} \frac{1}{\langle \zeta\rangle^3} d\zeta + \frac{1}{t^2} \int_{-t}^{\frac{t}{2}} \frac{1}{\langle \zeta \rangle} d\zeta + \frac{1}{t} \int_{\frac{t}{2}}^{t(1- 0.125\langle y\rangle^{-\frac{1}{2}})} \frac{1}{ (\zeta - t)^2}d\zeta + \frac{1}{t}\int_{t(1+0.125 \langle y\rangle^{-\frac{1}{2}})} ^\infty \frac{1}{( \zeta - t)^{2}} d\zeta\\
&\lesssim \frac{\log(t+2) + \langle y\rangle^{\frac{1}{2}}}{t^2}.
\end{align}
Similarly, we have $|\phi_{y,outer}| \lesssim \frac{ \langle y\rangle^\frac{1}{2}}{t}$.\\

\noindent \textbf{Step 3} (\textit{inner integrals}): For $\phi_{w,inner}$, we integrate the phase by parts in $\zeta$ twice:
\[
\phi_{w,inner} = - \frac{i}{(2\pi)^3}\iint_{\mathbb R^2} \xi^2 e^{i \frac{1}{\xi} h(\zeta)} \partial_\zeta \left[ \frac{1}{h'(\zeta)} \partial_\zeta\left[ \frac{1}{h'(\zeta)} e^{i \xi (z + \zeta y)} \chi_{inner}(t,y,\zeta) m^w_t(\zeta) v(\xi,\zeta)\right]\right] d\xi d\zeta.
\]
Hence, since $h$ and its derivatives are all bounded, we have
\begin{align}\label{eq:product_rule}
|\phi_{w,inner}| \lesssim \sum_{j=0}^2 \iint_{\mathbb R^2}\xi^2 \left| \partial_\zeta^{2-j} \frac{1}{h'(\zeta)^2}\right| |\partial_\zeta^j [ e^{i \xi (z + \zeta y)} \chi_{inner}(t,y,\zeta) m^w_t(\zeta) v(\xi,\zeta)]|d\xi d\zeta.
\end{align}

Within the region $|\zeta -t | \leq 0.25 t\langle y\rangle^{-\frac{1}{2}}$, we may write
\[
\frac{1}{h'(\zeta)} =  (1+ (\zeta - t)^2 ) g_t\left(\frac{\zeta}{t}\right), \qquad g_t(\sigma) \coloneq \frac{ \frac{1}{t^2} + \sigma^2 }{1 -2\sigma}.
\]
Observe that the Taylor expansion for $g_t(\sigma)$  centered $\sigma = 1$ has a radius of convergence $0.5$, and all coefficients can be bounded independently of $t$. For $\sigma = \frac{\zeta}{t}$, we have $|\sigma- 1| < 0.25$. Hence, for all such $\zeta$, we have
\[
\left| \partial_\zeta^{2-j} \frac{1}{h'(\zeta)^2} \right|^2 \lesssim \langle \zeta-t\rangle^{2+j}.
\]
As for derivatives on the product $e^{i \xi (z + \zeta y)} \chi_{inner}(t,y,\zeta) m^w_t(\zeta) v(\xi,\zeta)]$, we bound each factor appearing in the product rule.  On the support of $\chi_{inner}$, 
\begin{align}
|\partial_\zeta^j e^{i \xi (z + \zeta y)}| &\leq |\xi y|^j, \\
|\partial_\zeta^j \chi_{inner} | &\lesssim \frac{1}{\langle \zeta - t\rangle^j}, \\
|\partial_\zeta^j m^z| &\lesssim \frac{1}{\langle \zeta - t\rangle^{j+2}}, \\
|\partial_\zeta^j m^y| &\lesssim \frac{t}{\langle \zeta - t\rangle^{j+2}}. 
\end{align}
Applying the above to \eqref{eq:product_rule} and collecting terms, and then applying \eqref{eq:aux_rescaling_estimate0},
\begin{align}
&|\phi_{z,inner}|  +\frac{1}{t} |\phi_{y,inner}| \label{eq:two_birds_with_one_stone}\\
& \lesssim \sum_{0 \leq k \leq j \leq 2}\int_{t(1- 0.25 \langle y\rangle^{-\frac{1}{2}})}^{t(1+ 0.25 \langle y\rangle^{-\frac{1}{2}})} \int_{-\infty}^\infty\xi^2 \langle\zeta - t\rangle^{2+ j} (|\xi y| + \langle \zeta - t\rangle^{-1})^{j-k} |\partial_\zeta^{k} v(\xi,\zeta)|d\xi d\zeta \\
&\leq \sum_{0 \leq k \leq j \leq 2}\int_{t(1- 0.25 \langle y\rangle^{-\frac{1}{2}})}^{t(1+ 0.25 \langle y\rangle^{-\frac{1}{2}})} \int_{-\infty}^\infty\langle |\xi|^{2+k} \zeta - t\rangle^{2+ j} (|\xi y| + \langle \zeta - t\rangle^{-1})^{j-k} \langle \xi \langle \zeta\rangle\rangle^{-6}d\xi d\zeta.
\end{align}
We simplify the sum by bounding the expression by the endpoint terms:
\begin{align}
&\sum_{0\leq k \leq j\leq 2} \xi^{2+k} \langle\zeta - t\rangle^{2+ j} (|\xi y| + \langle \zeta - t\rangle^{-1})^{j-k} \\
&\lesssim  \sum_{0\leq j\leq 2} \xi^{2} \langle\zeta - t\rangle^{2+ j} (|\xi y| + \langle \zeta - t\rangle^{-1})^{j} +  \xi^{2+j} \langle\zeta - t\rangle^{2+ j} \\
&\lesssim \sum_{0\leq j \leq 2} \xi^2 \langle \zeta -t \rangle^2 (\langle \zeta - t\rangle |\xi y| + 1)^j +|\xi|^{2+ j} \langle \zeta - t\rangle^{2+j}\\
&\lesssim \sum_{0\leq j \leq 2}\xi^2 (\langle \zeta - t\rangle \langle y\rangle |\xi| + 1)^j\\
&\lesssim \xi^2 \langle\langle \zeta - t\rangle \langle y\rangle \xi^2\rangle^2.
\end{align}
Thus, continuing the estimate on \eqref{eq:two_birds_with_one_stone}, we have
\begin{align}
|\phi_{z,inner}|  +\frac{1}{t} |\phi_{y,inner}| &\leq \int_{t(1- 0.25 \langle y\rangle^{-\frac{1}{2}})}^{t(1+ 0.25 \langle y\rangle^{-\frac{1}{2}})} \int_{-\infty}^\infty \xi^2 \langle\langle \zeta - t\rangle \langle y\rangle \xi\rangle^2 \langle \xi \langle \zeta\rangle\rangle^{-6}d\xi d\zeta\\
&\lesssim  \int_{t(1- 0.25 \langle y\rangle^{-\frac{1}{2}})}^{t(1+ 0.25 \langle y\rangle^{-\frac{1}{2}})} \frac{1}{\langle \zeta\rangle^3} + \frac{\langle\zeta-t\rangle^2 \langle y\rangle^2}{\langle \zeta\rangle^5} d\zeta \\
&\lesssim \frac{\langle y\rangle^{\frac{1}{2}}}{t^2}.
\end{align}
This completes the proof. \qed

\vspace{2 mm}

\noindent \textbf{Acknowledgements:} JB gratefully acknowledges support from NSF DMS-2108633.  PF gratefully acknowledges support from NSF DMS-1928930, which was received while PF was in residence at the Simons Laufer Mathematical Sciences Institute in Berkeley, California, during the fall of 2025. SI gratefully acknowledges support from NSF DMS-2306528 and NSF CAREER DMS-2442781. The authors thank Xuecheng Wang for discussions. 

\vspace{2 mm}

\noindent \textbf{AI Declaration: } The core idea of this work was developed by the authors. Publicly available AI models were utilized to suggest refinements to our argument. All AI suggestions were carefully reviewed before implementing, and the entire content of the paper was written by the authors.

\bibliographystyle{plain} % We choose the "plain" reference style
\bibliography{bibliography} % Entries are in the refs.bib file

\appendix

\end{document}